\documentclass[12pt]{amsart}
\usepackage{latexsym}
\usepackage{amsmath,amsthm,amssymb, amscd,color}
\usepackage{mathrsfs}
\usepackage[all]{xy}
\usepackage[all]{xy}
\usepackage{tikz}
\usepackage{tikz-cd}
\usepackage{adjustbox}
\usetikzlibrary{snakes}
\usepackage[colorlinks=true, linkcolor=blue, citecolor=blue, urlcolor=blue, breaklinks=true]{hyperref}
\usepackage[capitalise]{cleveref}

\usepackage{blkarray} 
\usepackage{framed}

\theoremstyle{plain}
\newtheorem{theorem}{Theorem}[section]
\newtheorem{lemma}[theorem]{Lemma}
\newtheorem{cor}[theorem]{Corollary}
\numberwithin{equation}{section}

\theoremstyle{definition}
\newtheorem{defn}[theorem]{Definition}
\newtheorem{example}[theorem]{Example}

\newtheorem{prop}[theorem]{Proposition}
\newtheorem{remark}[theorem]{Remark}

\theoremstyle{remark}

\newcommand{\CC}{\mathbb{C}}

\newcommand{\RR}{\mathbb{R}}
\newcommand{\FF}{\mathbb{F}}
\newcommand{\ZZ}{\mathbb{Z}}

\newcommand{\G}{\mbox{Gr}}

\newcommand{\TC}{\mathrm{TC}}
\newcommand{\zl}{\mathrm{zcl}}
\newcommand{\ct}{\mathrm{cat}}
\newcommand{\cl}{\mathrm{cup}}

\begin{document}
\pagestyle{plain}

\title{On generalized complex projective product spaces}


\author[N. Daundkar]{Navnath Daundkar}
\address{Department of Mathematics, Indian Institute of Technology Madras, Chennai, India.}
\email{navnath@iitm.ac.in}

\author[S. Sarkar]{Soumen Sarkar}
\address{Department of Mathematics, Indian Institute of Technology Madras, Chennai-600036, India}
\email{soumen@iitm.ac.in}

\date{\today}

\maketitle

\begin{abstract}
Free circle action on manifolds has been explored in several articles. Gonzalez and Velasco considered some free circle actions on the finite product of spheres. In this paper, we introduce generalized complex projective product spaces, extending their definition and the concept of Dold manifolds. This produces infinitely many different classes of new smooth manifolds. First, we study the integral cohomology rings and stable tangent bundles on certain generalized complex-projective product spaces. Then, we discuss the product inequality for strongly equivaraint TC.
We exhibit the closeness of the lower and upper bounds for the LS-category and topological complexity of several classes of generalized complex projective product spaces. In many
cases, we compute the exact value of the LS-category.
\end{abstract}


\section{Introduction} \label{sec:1}

Let $S^1$ be the unit circle equipped with its usual group structure. Circle actions on manifolds have been studied extensively in various developments; see, for example,  \cite{Fin}, \cite{Kol}, \cite{Plo},  \cite{Ray}. More generally, torus actions on various kinds of manifolds and orbifolds have played a fundamental role in symplectic geometry \cite{Aud} and toric topology \cite{BP, DJ}. These actions provide a natural way of constructing and studying manifolds with rich geometric and topological structures.

The quotients of products of finitely many manifolds by circle actions provide a valuable source of manifolds that simultaneously reflect the topology of the factors and the geometry of the group action. A classical and particularly fruitful example is given by complex projective spaces, obtained as quotients of odd-dimensional spheres by the standard free $S^1$-action. More generally, diagonal circle actions on products of odd-dimensional spheres lead to the class of complex projective product spaces introduced by Gonz\'alez and Velasco \cite{Gonzalez-Velasco}. These manifolds have attracted considerable interest in algebraic and geometric topology, in part because their topology is apparently explicit to allow detailed calculations while still exhibiting interesting interactions between the topology of the constituent spaces.

The purpose of this paper is to extend this class of manifolds in a broader framework. Let $M$ and $N$ be smooth manifolds equipped with $S^1$-actions. We consider the diagonal $S^1$-action on $M\times N$ defined by
\[
\omega\cdot(x,y)
=
(\omega\cdot x,\omega\cdot y),
\qquad \omega\in S^1.
\]
If the $S^1$-action on either $M$ or $N$ is free, then the diagonal action on $M\times N$ is free. Then the orbit space
\[
C(M,N):=(M\times N)/S^1
\]
is a smooth manifold. We call such a manifold a \emph{generalized complex projective product space}.

The additional flexibility in this construction is useful from a topological point of view. Suppose, for instance, that the $S^1$-action on $N$ is free. Then the projection
\[
M\times N\longrightarrow N
\]
is $S^1$-equivariant and induces a fibre bundle
\[
M\longrightarrow C(M,N)\longrightarrow N/S^1.
\]
Similarly, if the action on $M$ is free, the projection onto the first factor induces a fibre bundle
\[
N\longrightarrow C(M,N)\longrightarrow M/S^1.
\]
Thus, whenever the actions on both factors are free, the generalized complex projective product space admits two natural fibre-bundle descriptions. In the case where $M, N$ are product of odd-dimensional spheres with the standard diagonal action, the corresponding orbit space is a complex projective product space. Moreover, whenever the action on one of the factors is free, the quotient map
$$M\times N\longrightarrow C(M,N)$$
is a principal $S^1$-bundle. These bundle structures provide useful frameworks for studying how classical invariants behave under the quotient construction.

A central theme of this paper is the study of topological invariants of generalized complex projective product spaces. We investigate, in particular, their cohomology rings, stable complex structure, Lusternik--Schnirelmann category, and topological complexity. These invariants capture different aspects of the topology of the resulting manifolds. The cohomology ring provides algebraic information about their global topology, while the stable complex structure computes the total Chern class. The Lusternik--Schnirelmann category and topological complexity are homotopy-theoretic invariants that quantify, respectively, the complexity of constructing maps into a space and the complexity of motion planning on the space.

The quotient description makes these questions particularly interesting. Even when the factors $M$ and $N$ have well-understood topology, the topology of $C(M, N)$ depends essentially on the circle action although quite complicated in general. Consequently, standard product formulas do not directly apply, and one is led naturally to combine tools from the topology of fibre bundles, transformation groups, characteristic classes, and the theory of sectional category and topological complexity. One of the objectives of this paper is to exploit this structure to obtain general results that can subsequently be specialized to the various families of examples described above.

The paper is organized as follows. In \Cref{sec:grass-toric}, we recall the definitions of several well-known manifolds equipped with torus actions. We then describe suitable circle subgroups of these torus actions and the induced circle actions on the manifolds.

In \Cref{sec:com_proj_prod_sps}, we introduce generalized complex projective product spaces and describe several important classes of examples, including complex projective product spaces, quotients involving quasitoric manifolds, Grassmannians, and moment-angle manifolds, see \Cref{ex:proj_prod_sp}, \Cref{ex:toric}, and \Cref{ex:grassmann}..

In \Cref{sec:cohomology}, we compute the integral cohomology rings of several classes of generalized complex projective product spaces. These computations provide the algebraic foundation for the study of their geometric and homotopy-theoretic invariants in the subsequent sections.

In \Cref{sec:tangent}, we describe canonical line bundles associated to the generalized complex projective product spaces introduced in \Cref{ex:proj_prod_sp}, \Cref{ex:toric}, and \Cref{ex:grassmann}. We use these bundles to study the stable tangent bundles of these manifolds and show that the corresponding manifolds admit stable complex structures. We then compute their total Chern classes.

Finally, in \Cref{sec-LSCat-TC}, we study the Lusternik--Schnirelmann category and the topological complexity of several classes of generalized complex projective product spaces. We first establish a product inequality for strongly equivariant topological complexity. We then obtain upper and lower bounds for the LS-category and topological complexity of various families of complex projective product spaces and investigate the extent to which these bounds coincide. In several cases, we determine the exact values of the LS-category and the topological complexity.

\section{Grassmann manifolds, quasitoric manifolds and moment angle manifolds}\label{sec:grass-toric}
\subsection{Grassmann manifolds}

Let $d, n \in \ZZ$ such that $0< d< n$. The set of all $d$-dimensional subspaces, denoted by $Gr_d(\CC^n)$, of $\CC^n$ is called a complex Grassmann manifold. The topology on $Gr_d(\CC^n)$ is given as follows. Let $M_d(\CC^n)$ be the set of all complex $n \times d$ matrix of rank $d$ and $\mbox{GL}(d,\CC)$ the set of all non-singular complex matrix of order $d$. Then $M_d(\CC^n)$ is an open subset of $(\CC^n)^d$. Observe that if $A \in M_d(\CC^n)$, then the span of the column vectors of $A$ (called the column space of $A$)  determines a $d$-dimensional plane in $\CC^n$. Suppose that $A, B \in M_d(\CC^n)$. Then the column space of $A$ is equal to the column space of $B$ if and only if there exists $T \in \mbox{GL}(d,\CC)$ such that $B=AT$. This defines a right $\mbox{GL}(d,\mathbb{C})$-action on $M_d(\CC^n)$. The orbit space $M_d(\CC^n)/ \mbox{GL}(d,\CC)$ can naturally be identified with $Gr_d(\CC^n)$. We denote the orbit map by
\begin{equation}\label{eq_map_pi}
\pi \colon M_d(\CC^{n})\to Gr_d(\CC^n).
\end{equation}
The topology on $Gr_d(\CC^n)$ is the quotient topology via $\pi$. It is well known that the space $Gr_d(\CC^n)$ is a $d(n-d)$-dimensional smooth manifold. Several basic properties, such as manifold and CW structure of $Gr_d(\CC^n)$, can be found in \cite{MiSt}. 

Let $T^n := \{(t_1,t_2,\dots,t_n) \in (\CC^{*})^n ~|~ |t_i|=1 ~\mbox{for}~ i=1, \ldots. n\}$ be the $n$-torus and $A=({\bf a}_1,{\bf a}_2,\dots ,{\bf a}_n)^{tr} \in M_d(\CC^n)$. Then $T^n$ acts linearly on $M_d(\CC^n)$ defined by  
\begin{equation}\label{T_act_Gsm}
(t_1, \ldots, t_n)({\bf a}_1,{\bf a}_2,\dots ,{\bf a}_n)^{tr}:=(t_1{\bf a}_1,t_2{\bf a}_2, \dots ,t_n{\bf a}_n)^{tr}.
\end{equation}
This action induces a natural $T^n$-action on $Gr_d(\CC^n)$ such that the orbit map $\pi$ of \eqref{eq_map_pi} is $T^n$-equivariant. In particular, if $S^1$ is a circle subgroup of $T^n$, then $S^1$ acts on $Gr_d(\CC^n)$. By a $S^1$-action on $Gr_d(\CC^n)$, we mean one of such actions. 

\subsection{Quasitoric manifolds}\label{subsec-qt}

Toric manifolds and moment angle manifolds were introduced and studied in the pioneering paper \cite{DJ}. These manifolds are topological generalizations of smooth projective toric varieties, and they are now known as quasitoric manifolds. Here, we recall the definition of these manifolds following \cite{DJ}.

An $n$-dimensional simple polytope in $\RR^n$ is a convex polytope where exactly $n$ bounding hyperplanes meet at each vertex. For example, the $n$-simplex, the $n$-cube, and their finite Cartesian products are simple polytopes. Let $Q$ be a simple polytope. Then zero-dimensional faces of $Q$ are called vertices, denoted by $V(Q)$, and codimension one faces of $Q$ are called facets, denoted by $\mathcal{F}(Q)$. We denote  $T:=S^1 = \{z \in \CC : |z|=1\}$.
\begin{defn}
A smooth action of $T^n$ on a $2n$-dimensional smooth manifold $M^{2n}$ is said to be locally standard if every point $y \in M^{2n} $ has a $T^n$-invariant open neighborhood $U_y$ and a diffeomorphism $\psi_y \colon U_y \to V$, where $V$ is a $T^n$-invariant open subset of $\CC^n$, and an isomorphism $\delta_y \colon T^n \to T^n$ such that  $\psi_y (t\cdot x) = \delta_y (t) \cdot \psi_y(x)$ for all $(t,x) \in T^n \times U_y$.
\end{defn}

We recall that such a map $\psi_y$ is known as a weakly equivariant map and, in addition, if $\delta_y$ is identity then it is called an equivariant map, or also a $T^n$-equivariant map to emphasize the group action.

\begin{defn}\label{qtd02}
A closed smooth $2n$-dimensional $T^n$-manifold $M^{2n}$ is called a quasitoric manifold over a simple polytope $Q$ if the following conditions are satisfied:
\begin{enumerate}
\item the $T^n$ action is locally standard, denoted by $\rho$.
\item the orbit map $\mathfrak{q} \colon M^{2n} \to Q$ sends  an $\ell$-dimensional orbit to a point in the interior of an $\ell$-dimensional face of $Q$.
\end{enumerate}
\end{defn}

\begin{example}
 All complex projective spaces and their finite products are quasitoric manifolds. 
\end{example}
Each circle subgroup of $T^n$ gives a circle action on $M^{2n}$. Therefore, any $T^n$-invariant subset of $M^{2n}$ is also invariant under this circle action. We note that the product of two quasitoric manifolds is again a quasitoric manifold.

\subsection{Moment angle manifolds}\label{moment-angle}
In this subsection, we recall the concept of moment angle manifolds which are also central objects in toric topology. We shall follow \cite[Subsection 4.1]{DJ} and show how they are related to the quasitoric manifold. Let $F_1, \ldots, F_\mu$ be the facets of an $n$-dimensional simple polytope $Q$ and $G_i$ be the subgroup of $T^{\mu}$ generated by the $i$th factor for $i=1, \ldots, \mu$. If $F$ is  a proper face of $Q$ of codimension-$r$, then $F = F_{i_1} \cap \cdots \cap F_{i_r}$ for a unique collection of facets $F_{i_1}, \ldots, F_{i_r}$. Let $T_F$ be the subgroup of $T^{\mu}$ generated by $\{ G_{i_1}, \ldots, G_{i_r}\}$. We fix $T_{Q} = \{1\} \in T^{\mu}$.
We define an equivalence relation $\sim$ on the product $T^{\mu} \times Q$ as follows,
\begin{equation}\label{defeqiv}
(s, x) \sim (t, y) ~ \mbox{if and only if}~  x = y ~ \mbox{and} ~ ts^{-1} \in T_{F}
\end{equation}
where $ F \subseteq Q $ is the unique face containing the point $ x $ in its relative interior. Then the identification space $$Z_Q := (T^{\mu} \times Q)/\sim $$ is a manifold. This is called a moment angle manifold on $Q$. Also $Z_Q$ is a $T^{\mu}$-space. We refer to \cite[Section 6]{BP} for a different construction and for further properties of moment angle complexes.

For the rest of this section, we discuss a relation between a toric manifold $M^{2n}$ over $Q$ and the moment angle manifold $Z_Q$. Let  $\mathfrak{q} \colon M^{2n} \to Q$ be the orbit map of a $T^n$-manifold $M^{2n}$ and  $\mathcal{F}(Q) =\{ F_1, \ldots, F_\mu\}$ the facets of $Q$. So the subset $\mathfrak{q}^{-1}(F_i) $ is fixed by a circle subgroup $S_i^1 \subseteq T^n$ for $i=1, \ldots, \mu$. The assignment
\begin{equation}\label{eq:char_func2}
 F_i \to S^1_i 
\end{equation}
 for $i=1, \ldots , \mu$ is known as the characteristic function of $M^{2n}$, see \cite[(5.4)]{BP}. Note that the circle subgroup $S^1_i$ is uniquely determined by a primitive element $(\lambda_{i_1}, \ldots, \lambda_{i_n}) \in \ZZ^n$ up to sign, where $\ZZ^n$ is the integral lattice in the Lie algebra of $T^n$. The assignment  $\Lambda \colon F_i \to \lambda_i=(\lambda_{i_1}, \ldots, \lambda_{i_n})$ 
 for $i=1, \ldots , \mu$ is also known as the characteristic function on $Q$. This assignment induces the surjective map, also denoted by $\Lambda$, $\Lambda \colon \ZZ^{\mu} \to \ZZ^n$ defined by $e_i \mapsto \lambda_i$ for  $i=1, \ldots, \mu$ if $\{e_1, \ldots, e_{\mu}\}$ is the standard basis of $\ZZ^{\mu}$ over $\mathbb{Z}$ as a module.  Therefore, by \Cref{qtd02}, one gets the following short exact sequence of Lie groups
\begin{equation}\label{eq:torus_exact}
0 \to \ker({\rm exp}\Lambda) \to T^{\mu} \xrightarrow{{\rm exp}\Lambda} T^n \to 0. 
\end{equation} 
\begin{prop}\cite[Proposition 6.5]{BP}\label{prop:moment_quasitoric}  The group
$\ker({\rm exp}\Lambda)$ is an $(m-n)$-dimensional torus subgroup of $T^{\mu}$ and it acts freely on $Z_Q$ with $Z_Q/ \ker({\rm exp}\Lambda) \cong M^{2n}$. 
\end{prop}
Throughout the paper, for simplicity, we let $T^{\mu-n}:=\ker({\rm exp}\Lambda)$.
We note that any odd-dimensional sphere and finite products of odd-dimensional spheres are moment angle manifolds.  Since $\mu > n$, there are circle subgroups of $T^{\mu -n}$. Each of these circle subgroups acts freely on $Z_Q$. We note that the product of two moment angle manifolds is again a moment angle manifold.

\section{Generalized complex projective product spaces}\label{sec:com_proj_prod_sps}
In this section, we discuss some examples of generalized complex projective product spaces. Of course, one can construct many other classes of these spaces, but our interests will focus on the spaces of the following examples. In these examples, we use the fact that if $\sigma$ is a circle action on $X$ and $\tau$ is a free circle action on $Y$, then $\sigma\times \tau$ is a free circle action on $X\times Y$.

\begin{example}\label{ex:proj_prod_sp}
Let $$ S^{2m+1}=\bigg\{(z_1, \ldots, z_{m+1}) \in \CC^{m+1} ~\big |~ \sum_{s=1}^{m+1} |z_s|^2=1\bigg\}$$ be the (2m+1)-dimensional sphere for the non-negative integer $m$. The standard circle action on $S^{2m+1}$ is defined by $$\omega \cdot (z_1,\dots, z_{m+1}) \to (\omega z_1, \dots,  \omega z_{m+1})$$ for $\omega \in S^1$. 
Let $S^{2m_i+1}$ and $S^{2n_j+1}$ be spheres for $i=1, \dots, k $ and $j=1, \ldots, \ell$. We consider the standard $S^1$-action, denoted by $\tau_i$, on $S^{2m_i+1}$ for $i=1, \ldots, k$.  Then the diagonal $S^1$-action $\tau_1 \times \cdots \times \tau_k$ on the product $S^{2m_1+1} \times \cdots \times S^{2m_k+1}$ is free. We denote $\tau (k) := \tau_1 \times \cdots \times \tau_k$, $\overline{m}(k)=(m_1,\dots, m_k)$ and $S(\overline{m}(k))=S^{2m_1+1} \times \cdots \times S^{2m_k+1}$. 
Then the orbit space 
$$C_{\overline{m}(k)} := S(\overline{m}(k))/\tau(k)$$ 
is called a \emph{complex projective product space}, introduced by Gonz\'{a}lez and Velasco in \cite{Gonzalez-Velasco}. This space is a manifold of dimension $2(n_1+\dots+n_k)+k-1$. 

Now, we consider the circle action $\sigma_j$ on the $(2n_j+1)$-dimensional sphere
$S^{2n_j+1}$ defined, for 
\ $1 \leq j \leq \ell$, \ by 
\begin{equation}\label{eq:multiple_reflection}
\sigma_j \colon (z'_1, \ldots, z'_{p_j}, z'_{p_j+1}, \ldots, z'_{n_j+1}) \mapsto (z'_1, \ldots, z'_{p_j}, \omega z'_{p_j+1}, \ldots, \omega z'_{n_j+1})
\end{equation}
for $\omega \in S^1$ and some $0 \leq p_j \leq n_j$. So $S^1$ acts on the product space $S(\overline{n}(\ell))= S^{2n_1+1} \times \cdots \times S^{2n_\ell+1}$ via the diagonal action $\sigma(\ell) =\sigma_1 \times \cdots \times \sigma_{\ell}$. 
Thus we have the free $S^1$-action on  the product $S(\overline{n}(\ell))\times S((\overline{m}(k))$ defined by 
\begin{equation}\label{eq:multiple_reflection2}
(({\bf z_1}', \ldots, {\bf z_\ell}'), ({\bf z}_1, \ldots, {\bf z}_{k})) \mapsto ((\sigma(\ell)({\bf z'_1}, \ldots, {\bf z'_\ell}), \tau(k)({\bf z}_1, \ldots, {\bf z}_{k})),
\end{equation}
where $\sigma(\ell)({\bf z'_1}, \ldots, {\bf z'_\ell}) =(\sigma_1({\bf z'_1}), \ldots, \sigma_\ell({\bf z'_\ell})$, $\tau(k)({\bf z}_1, \ldots, {\bf z}_{k})=(\tau_1({\bf z}_1), \ldots, \tau_k({\bf z}_{k}))$. We denote the orbit space
by   $C_{(\overline{n}, \overline{p})(\ell),\overline{m}(k)}$ where $\overline{m}(k)=(m_1, \ldots, m_k)$ and $(\overline{n}, \overline{p})(\ell)= ((n_1, p_1), \ldots, (n_\ell, p_\ell))$.

More generally,  let $N$ be a smooth manifold with a  free circle action denoted by $\tau$. Thus, we have a $S^1$-action on the product $S(\overline{n}(\ell))\times N$ defined by \[(z_1,\dots, z_{\ell}, y) \mapsto (\sigma(\ell)(z_1, \dots, z_{\ell}),\tau(y)).\]
Then the orbit space of this action is denoted by $C((\overline{n}, \overline{p})(\ell), N)$, is a generalized complex projective product space.
\end{example}

\begin{example}\label{ex:toric}
Let $M^{2n}$ be a $2n$-dimensional quasitoric manifold and $T^n$ acts on $M^{2n}$ as in \Cref{qtd02}. Let $S^1$ be a circle subgroup of $T^n$ and $\rho$ the induced $S^1$ action on $M^{2n}$. 
Then one can define a $S^1$-action on $M^{2n}\times S(\overline{m}(k)) $ by $$(y, {\bf z}_1, \ldots, {\bf z}_k) \mapsto (\rho (y), {\tau_1 (\bf z}_1), \ldots, \tau_k({\bf z}_k)).$$ This action is free, and the orbit space is a generalized complex projective product space.  We denote it by $C(M^{2n}, S(\overline{m}(k))$. 
 Note that the orbit map
\begin{equation}\label{double2toric}
M^{2n} \times S(\overline{m}(k)) \longrightarrow C( M^{2n},S(\overline{m}(k))
\end{equation}
 is a principal $S^1$-bundle. The projection  $M^{2n} \times S(\overline{m}(k)) \to S(\overline{m}(k))$ induces a smooth fibre bundle: 
\begin{equation}\label{fiber2toric}
M^{2n} \longrightarrow C(M^{2n}, S(\overline{m}(k))) \longrightarrow  C_{\overline{m}(k)}.
\end{equation}

Moreover, if $N$ is a manifold with free $S^1$-action then $C(M^{2n}, N) = (M^{2n}\times N)/S^1$ is called a toric-projective product space. 

In particular, if $N$ is a moment angle manifold $Z_Q$ with a free $S^1$-action, then the orbit space $(M^{2n}\times Z_Q)/S^1$ is a {\em toric-projective product space}.
\end{example}

\begin{example}\label{ex:grassmann}
Let $Gr_d(\CC^n)$ be a complex Grassman manifold with a circle action, denoted by $\xi$, as discussed in \Cref{sec:grass-toric}. Then one can define a $S^1$-action on the product $ Gr_d(\CC^n)\times S(\overline{m}(k))$ by 
$$(x, {\bf z}_1, \ldots, {\bf z}_k ) \mapsto ( \xi (x), \tau(k)({\bf z}_1, \ldots, {\bf z}_k)).$$ 
This action is free, and the orbit space is a generalized complex projective product space.  We denote it by $C(Gr_d(\CC^n), S(\overline{m}(k))$, and call Grassmann projective product spaces.  Note that the orbit map
\begin{equation}\label{double2}
 Gr_d(\CC^n)\times S(\overline{m}(k)) \longrightarrow C(Gr_d(\CC^n), S(\overline{m}(k)))
\end{equation}
 is a principal $S^1$-bundle. 
 The projection  $Gr_d(\CC^n)\times S(\overline{m}(k))  \to  S(\overline{m}(k))$ induces a smooth fiber bundle: 
\begin{equation}\label{fiber2}
Gr_d(\CC^n) \longrightarrow C(Gr_d(\CC^n), S(\overline{m}(k))) \longrightarrow C_{\overline{m}(k)}.
\end{equation}

Moreover, if $N$ is a manifold with free $S^1$-action then the space $C(Gr_d(\CC^n), N) = ( Gr_d(\CC^n) \times N)/S^1$ is a generalized complex projective product space. For example, consider a free $S^1$-action on the moment angle manifold $Z_Q$, then the orbit space $( Gr_{d_1}(\CC^{n_1}) \times \ldots \times Gr_{d_r}(\CC^{n_r}) \times Z_Q)/S^1$ is a generalized complex projective product space.
\end{example}

\begin{example}\label{ex:moment_angle}
Let $M_1, \ldots, M_k$ be moment angle manifolds corresponding to the simple polytopes $Q_1, \ldots, Q_k$ respectively. Then there is a free $S^1$ action on each $M_i$ for $i=1, \ldots, k$ by \Cref{sec:grass-toric}. Then $S^1$ acts freely on the product $M_1 \times \cdots \times M_k$ via the diagonal action. So the orbit space $(M_1 \times \cdots \times M_k)/S^1$ is a generalized complex projective product space.  We remark that the finite product of moment angle manifolds is again a moment angle manifold. So complex projective product space $C_{\overline{m}(k)}$ of \Cref{ex:proj_prod_sp}  belongs to this class of manifolds. These spaces are also known as partial quotients, see for example, \cite[Section 7.5]{Pan}, \cite{Fra}.   
\end{example}

\begin{example}
Let $Z_Q$ be a moment angle manifold over a simple polytope with a free circle action and $X$ be an almost complex manifold equipped with a circle action. Then $S^1$ acts freely on the product $X \times Z_Q$ via the diagonal action. So, the orbit space $(X \times Z_Q)/S^1$ is a generalized complex projective product space. Note that (partial) Flag manifolds and Hessenberg manifolds satisfy the condition on $X$. 
\end{example}

\section{Cohomology of generalized complex projective product spaces}\label{sec:cohomology}

In this section, we first compute the cohomology ring with $\ZZ$ coefficients and $\mathbb{Z}_2$ coefficients of the manifolds in \Cref{ex:proj_prod_sp}. Then we describe the cohomology ring of the generalized complex projective product space $C(M, S(\bar{m}(k)))$, where $M$ is a quasitoric manifold or Grassmann manifold with $\ZZ$ coefficients.

\subsection{Cohomology of spaces in \Cref{ex:proj_prod_sp}}\label{subsec:proj_prod_sp}

 In this subsection, we  show that $C_{\overline{m}(k)}$ is an odd-dimensional sphere bundle over $C_{\overline{m}(k-1)}$ for $k \geq 2$ where $C_{\overline{m}(1)}:=\CC P^{m_1}$. We also show that the manifold  $C_{(\overline{n}, \overline{p})(\ell),\overline{m}(k)}$ defined in  \Cref{ex:proj_prod_sp} is an iterated sphere bundle over $C_{\overline{m}(k)}$. Then we compute its cohomology ring with $\ZZ$ and $\ZZ_2$ coefficients.

We consider the trivial sphere bundle $$S^{2m_1+1} \times S^{2m_2+1} \xrightarrow{\tilde{\xi}}  S^{2m_1+1}.$$  The group $S^1$ acts on its total space  by
$$({\bf z}_1, {\bf z}_2) \mapsto (\tau_1({\bf x}_1),  \tau_2({\bf x}_2))$$
So $\tilde{\xi}$ is a $S^1$-equivariant map. Since $S^1$ acts freely on the base $S^{2m_1+1}$, $\tilde{\xi}$ induces a sphere bundle $$C_{\overline{m}(2)} \xrightarrow{\tilde{\xi}}  C_{\overline{m}(1)}$$ 
with fiber $S^{2m_2+1}$. By similar arguments, we can show that
\begin{equation}\label{eq:sphere_bundle1}
C_{\overline{m}(k)}  \to C_{\overline{m}(k-1)} 
\end{equation}
is a sphere bundle with fibre $S^{2m_k+1}$  for $k \geq 2$. We consider the trivial complex line bundle 
$$S(\overline{m}(k-1)) \times \CC \xrightarrow{}  S(\overline{m}(k-1)).$$  
The group $S^1$ acts on the total space  by $$({\bf z}_1, \ldots, {\bf z}_{k-1},  z) \mapsto (\tau_1({\bf z}_1), \ldots, \tau_k({\bf x}_{k-1}), t z ))$$ where $t \in S^1 \subset \CC$. 
This induces a complex line bundle 
\begin{equation}\label{eq:can_line_bundle1}
\psi_k \colon E  \to  C_{\overline{m}(k-1)}.
\end{equation} 
 So the bundle $E(m_k) : = (m_k +1)\psi_k$ is a (complex) vector bundle with fiber $\CC^{m_k +1}$  over $C_{\overline{m}(k-1)}$ for $k \geq 2$. From the action $\tau_k$ defined in \Cref{ex:proj_prod_sp}, we can conclude that the sphere bundle 
 $$S(m_k +1)\psi_k \to  C_{\overline{m}(k-1)}$$ 
 is the sphere bundle in \eqref{eq:sphere_bundle1}. We denote the associated disk bundle by $D(E(m_k))$ (or by $D(m_k +1)\psi_k$), for any $k \geq 2$.

Now we consider the trivial sphere bundle 
$$S(\overline{n}(\ell-1)) \times S^{2n_\ell+1}\times S(\overline{m}(k))\xrightarrow{\tilde{\xi}}  S(\overline{n}(\ell-1)) \times S(\overline{m}(k)).$$  
The group $S^1$ acts on its total space  by
$$( {\bf z'}_1,\dots, {\bf z'}_{\ell-1}, {\bf z'}_\ell, {\bf z}_1,\dots, {\bf z}_k) \mapsto ( \sigma(\ell)({\bf z'}_1,\dots, {\bf z'}_{\ell-1},  {\bf z'}_{\ell}), \tau(k)({\bf z}_1,\dots, {\bf z}_k))$$
and on its base by  
$$( {\bf z'}_1, \ldots, {\bf z'}_{\ell-1}, {\bf z}_1, \ldots, {\bf z}_k) \mapsto ( \sigma(\ell-1)({\bf z'}_1,\dots, {\bf z'}_{\ell-1}), \tau(k)({\bf z}_1,\dots, {\bf z}_k))$$  where the actions $\sigma_1, \ldots, \sigma_{\ell-1}$ and $\sigma_{\ell}$ are defined in \eqref{eq:multiple_reflection}. So the map $\tilde{\xi}$ is a $S^1$-equivariant map. Since the group $S^1$ acts freely on the base, $\tilde{\xi}$ induces the following sphere bundle  
\begin{equation}\label{eq:sphere_bundle}
C_{(\overline{n}, \overline{p})(\ell),\overline{m}(k)} \to C_{(\overline{n}, \overline{p})(\ell-1),\overline{m}(k)}    \end{equation} 
with fiber $S^{2n_{\ell}+1}$ for any $\ell \geq 2$. By similar arguments we can show that 
$$C_{(\overline{n}, \overline{p})(1),\overline{m}(k)}\to C_{\overline{m}(k)} $$ 
is a sphere bundle with fiber $S^{2n_1+1}$. 
For simplicity of notation, we assume that this bundle corresponds to $\ell =1$. 
Next,  we consider the trivial complex line bundle 
$$S(\overline{n}(\ell -1)) \times \CC \times S(\overline{m}(k)) \xrightarrow{}    S(\overline{n}(\ell -1)) \times S(\overline{m}(k)).$$  
The group $S^1$ acts on the total space  by 
$$( {\bf z'}_1, \ldots, {\bf z'}_{\ell-1}, z,{\bf z}_1, \ldots, {\bf z}_k) \mapsto (\sigma(\ell-1)({\bf z'}_1, \ldots, {\bf z'}_{\ell-1}), t z, \tau(k)({\bf z}_1, \ldots, {\bf z}_k) )$$ where $t \in S^1 \subset \CC$. This induces a complex line bundle 
\begin{equation}\label{eq:can_line_bundle}
\eta_\ell \colon E  \to  C_{(\overline{n}, \overline{p})(\ell-1),\overline{m}(k)} .
\end{equation} 
 So the bundle $E(n_\ell, p_\ell) : = p_\ell \varepsilon_{\CC} \oplus (n_\ell - p_\ell +1)\eta_\ell$ is a vector bundle with fiber  $\CC^{n_\ell +1}$ over $C_{(\overline{n}, \overline{p})(\ell-1),\overline{m}(k)}$ where $\varepsilon_{\CC}$ is the trivial complex line bundle. From the action $\sigma_\ell$ in \eqref{eq:multiple_reflection}, we can conclude that the sphere bundle 
 $$S(p_\ell \varepsilon_{\CC} \oplus (n_\ell - p_\ell +1)\eta_\ell) \to  C_{(\overline{n}, \overline{p})(\ell-1),\overline{m}(k)}$$ 
 is the sphere bundle in \eqref{eq:sphere_bundle}. We denote the associated disk bundle by $D(E(n_{\ell}, p_\ell))$ (or by $D(p_\ell \varepsilon_{\CC} \oplus (n_\ell - p_\ell +1)\eta_\ell)$), for any $\ell \geq 1$.

If $0 \neq \alpha \in H^q(X;\ZZ)$, then we write $|\alpha| = q$ for the degree of $\alpha$. 

\begin{prop}\label{prop_cohom_gen_proj_prod}
Let $m_1\leq \cdots \leq m_k\leq n_1 \cdots \leq n_\ell$, and $p_j \geq 1$ for $1 \leq j \leq \ell$. Then  $H^*(C_{(\overline{n}, \overline{p})(\ell),\overline{m}(k)}; \ZZ) $ is isomorphic as a graded ring to  
\begin{align*}
\ZZ[\alpha]/(\alpha^{m_1+1})\otimes\Lambda[\alpha_2, \ldots, \alpha_k, \beta_1,\ldots, \beta_\ell],
\end{align*} where $|\alpha|=2$, $|\alpha_i|=2m_i+1,$ for $2 \leq i \leq k $ and $|\beta_j|=2n_j+1,$ for $1 \leq j \leq \ell.$
Moreover, $\alpha_i^2=0$ for $2 \leq i \leq k$, and $\beta_j^2=0$ for $1 \leq j \leq \ell$. 
\end{prop}

\begin{proof}
Note that $C_{\overline{m}(1)} = \CC P^{m_1}$. Thus $H^*(C_{\overline{m}(1)}; \ZZ) \cong \ZZ[\alpha]/(\alpha^{m_1+1})$. Then it follows from \cite[Theorem 2.1]{Gonzalez-Velasco}  that $$H^*(C_{\overline{m}(k)}; \ZZ) \approx \ZZ[\alpha]/(\alpha^{m_1+1})\otimes\Lambda(\alpha_2, \ldots, \alpha_k)$$ as a graded ring,  and the corresponding relations among $\alpha, \alpha_2, \ldots, \alpha_k$ hold. So the result is true for $\ell=0$, which starts the proof by induction. 
We next prove the claim when $\ell =1$, i.e., for $H^*(C_{(\overline{n}, \overline{p})(1),\overline{m}(k)}; \ZZ)$.
 Since the inductive
step from $\ell$ to $\ell+1, \ \ell \geq 1$, is similar to that from $0$ to $1$, this will complete the proof. 

We have shown that $C_{(\overline{n}, \overline{p})(1),\overline{m}(k)} \cong S(p_1 \varepsilon_{\CC} \oplus (n_1+1-p_1)\eta_1)$ as a sphere bundle over $C_{\overline{m}(k)}$. This gives the cofibration $$C_{(\overline{n}, \overline{p})(1),\overline{m}(k)} \xrightarrow{q} C_{\overline{m}(k)} \xrightarrow{\iota} T(E(n_1,p_1)) \simeq M(q)$$
where the first map is given by $q([{\bf x}_1, \ldots, {\bf x}_k, {\bf y}_1]) =[{\bf x}_1, \ldots, {\bf x}_k]$, $T(E(n_1, p_1))$ denotes the Thom space of the bundle $E(n_1, p_1)$ defined after \eqref{eq:can_line_bundle}, and $M(q)$ is the mapping cone of $q$. Hence one gets the following long exact sequence with coefficients in $\ZZ$ 
$$  \to H^*(T(E(n_1, p_1)))  \xrightarrow{\iota^*}  H^*(C_{\overline{m}(k)})  \xrightarrow{q^*} H^*(C_{(\overline{n}, \overline{p})(1),\overline{m}(k)} ) \xrightarrow{\delta} H^{*+1}(T(E(n_1, p_1))) \to  .$$ 
Since our assumption is $m_1 \leq \cdots \leq m_k \leq n_1$, we have the map 
$$\phi \colon   C_{\overline{m}(k)} \to C_{(\overline{n}, \overline{p})(1),\overline{m}(k)}$$ 
defined by  $[{\bf x}_1, \ldots, {\bf x}_k] \mapsto [{\bf x}_1, \ldots, {\bf x}_k, {\bf x}_k]$. 
So the composition map $ q \circ \phi$ gives the identity on $ C_{\overline{m}(k)}$. 
Thus we get the splitting 
$$H^*( C_{(\overline{n}, \overline{p})(1),\overline{m}(k)}; \ZZ) \approx H^*( C_{\overline{m}(k)}; \ZZ) \oplus H^{*+1}(T(E(n_1, p_1)) ;\ZZ).$$  
Let $\beta_1$ be the image of the Thom class in $H^{n_1+1}(T(E(n_1,p_1)))$ under this isomorphism. Note that the Thom isomorphism  gives ${H^{*+1}(T(E(n_1,p_1)); \ZZ ) \approx H^*(C_{\overline m} ; \ZZ) \cdot \beta_1 }$. One can say that  
$$H^*(C_{(\overline{n}, \overline{p})(1),\overline{m}(k)}; \ZZ) \approx H^*(C_{\overline{m}(k)}; \ZZ) \oplus H^*(C_{\overline{m}(k)} ; \ZZ) \cdot \beta_1.$$
Now, by induction, we get the description of the cohomology ring $H^*(C_{(\overline{n}, \overline{p})(\ell),\overline{m}(k)}; \ZZ) $ as described. 

The claim  $\alpha_i^2=0$ for $2 \leq i \leq k$, and $\beta_j^2=0$ for $1 \leq j \leq \ell$ follows from the fact that $|\alpha_i|$ and $|\beta_j|$ are odd and $H^*(C_{(\overline{n}, \overline{p})(\ell),\overline{m}(k)}; \ZZ) $  has no torsion. 
\end{proof}

\begin{prop}\label{prop_cohom_gen_proj_prod2}
Let $m_1\leq \cdots \leq m_k\leq n_1 \cdots \leq n_\ell$, and $p_j \geq 1, 1 \leq j \leq \ell$. Then,  
\begin{equation}\label{eq-mod2}
H^*(C_{(\overline{n}, \overline{p})(\ell),\overline{m}(k)}; \ZZ_2) \cong \ZZ_2[\alpha]/(\alpha^{m_1+1})\otimes\Lambda[\alpha_2, \ldots, \alpha_k, \beta_1,\ldots, \beta_\ell],
\end{equation} where $|\alpha|=2$, $|\alpha_i|=2m_i+1,$ for $2 \leq i \leq k $ and $|\beta_j|=2n_j+1,$ for $1 \leq j \leq \ell,$ $Sq(\alpha_i)=(1+\alpha)^{m_i+1}\alpha_i,$ and $Sq(\beta_j)=(1+\alpha)^{n_j+1-p_j} \beta_j$ .

Otherwise,  $\alpha_i^2 = \alpha^{m_i}\alpha_i$ for all 
$i \geq 2$ with $m_i = m_1$, and $\beta_j^2 = \alpha^{n_j} \beta_j $ for those $j$ with $p_j = 1$ and $n_j = m_1$. 
\end{prop}

\begin{proof}
The claim \eqref{eq-mod2} follows using the singular arguments as in the proof of \Cref{prop_cohom_gen_proj_prod}. Now, 
the projection $pr \colon C_{(\overline{n}, \overline{p})(1),\overline{m}(k)} \to C_{m_1} \approx \CC P^{m_1}$ gives $pr^*(\eta_0)=\eta_1$ by naturality, where $\eta_0$ is the canonical line bundle on $\CC P^{m_1}$ and $\eta_{\ell}$ is defined in \eqref{eq:can_line_bundle} for $\ell \geq 1$. 
Therefore, the total Steenrod square and the total Stiefel-Whitney class have the following relation in our setting, using the arguments in \cite[Page 94]{MiSt}. 
\begin{align*}
Sq(\beta_1)&=W(p_1\varepsilon \oplus (n_1+1-p_1) \eta_1) \beta_1 \\ &= (1+\alpha)^{n_1+1-p_1}\beta_1,
\end{align*} 
where $\alpha = w_2(\eta_0)$, the canonical generator of $H^2(\CC P^{m_1}; \ZZ_2)$, and $|\beta_1| = 2n_1+1$. Then $\beta_1^2 = \binom{n_1+1-p_1}{n_1}\alpha^{n_1} \beta_1$ and hence $\beta_1^2$ is zero for $p_1>1$ and the result holds in this case.
For $p_1=1$, we get the same structure with $\beta_1^2=\alpha^{n_1} \beta_1.$
\end{proof}

\subsection{Cohomology of spaces in \Cref{ex:toric} and \Cref{ex:grassmann}}\label{subsec_cohom_toric}
First, we recall the cohomology ring of a quasitoric manifold and Grassmann manifold following \cite{DJ} and \cite{MiSt}, respectively. Then we compute the integral cohomology of quasitoric and Grassmann projective product spaces. 

Let $X$ be a $2n$-dimensional toric manifold, and $\mathfrak{q} \colon X \to Q$ the orbit map as in Definition \ref{qtd02}.
Let  $\mathcal{F}(Q):= \{F_1,\ldots,F_{\mu}\}$ be the facets of $Q$.
Recall the assignment  $\lambda(F_i) =\lambda_i :=(\lambda_{i_1}, \ldots, \lambda_{i_n}) \in \ZZ^n$ for each  $i \in \{1, \ldots, {\mu}\}$ determined by \eqref{eq:char_func2}. Let $I$ and $J$ be the ideals of $\ZZ[u_1,\ldots, u_{\mu}]$ generated by the sets 
\begin{equation}\label{eq:ideal_I-J}
\{u_{j_1}\cdots u_{j_k}\colon \cap_{i=1}^k F_{j_i} = \phi\} ~~\mbox{and} ~~ \{\lambda_{1_s}u_1+ \cdots + \lambda_{{\mu}_s}u_{\mu} \colon 1\leq s \leq n\}
\end{equation}
 respectively. The ideal $I$ is known as the Stanley-Reisner ideal. Then by  \cite[Theorem 4.14]{DJ}, we have the following.
\begin{equation}\label{eq:cohom_toric_mfds}
H^*(X; \ZZ) \approx \frac{\ZZ [u_1, \ldots, u_{\mu}]}{I+J}
\end{equation}\\
where $u_i$ is the Poincar\'e dual of $\mathfrak{q}^{-1}(F_i)$ for $i=1, \ldots,  {\mu}$.

We recall the cohomology ring description of Grassmann manifolds from \cite{MiSt}.
The integral cohomology ring of $Gr_d(\CC^n)$ is described as follows: \[H^*(Gr_d(\CC^n),\ZZ)=\frac{\ZZ[c_1,\dots,c_d]}{\left<h_{n-d-1},\dots,h_n\right>},\]
where $|c_i|=2i$ and $h_j$ is defined as the $2j$-th degree term in the series expansion of $(1+c_1+\dots+c_d)^{-1}$.

In the rest of this subsection, we compute the cohomology ring of the generalized complex projective product space $C(X, S(\overline{m}(k))$ with $\ZZ$ coefficients, when $X$ is either a quasitoric manifold or a Grassmann manifold.

\begin{theorem}\label{thm:toric_mod2}
Let $m_1, \ldots, m_k$ be positive integers greater than one. Let $X$ either be a quasitoric manifold $M^{2n}$ or a Grassmann manifold $Gr_d(\CC^n)$. Then
$$H^*(C(X, S(\overline{m}(k))); \ZZ) \approx H^*(C_{\overline{m}(k)};\ZZ) \otimes H^*(X; \ZZ).$$ 
\end{theorem}  
\begin{proof}

By hypothesis $\pi_1(C_{\overline{m}(k)}) = 0$. 
Note that the cohomology groups of the fibre and base of this bundle have finite dimensions over the field $\ZZ$. Also, both the fibre $X$ and the base $C_{\overline{m}(k)}$ are path connected. By \cite[Theorem 3.1]{DJ}, $H^*(X; \ZZ)$ is concentrated in even degrees. So, by applying \cite[Proposition 5.5]{Mcc} on the bundle \eqref{fiber2}  one gets that the corresponding spectral sequence collapses at $E_2$. Hence $X$ is totally non-homologous to zero in $C(X, S(\overline{m}(k)))$ with respect to $\ZZ$. Thus by \cite[Theorem 5.10]{Mcc}, one gets the result. 
\end{proof}

\section{Tangent bundles of some generalized complex projective product spaces}\label{sec:tangent}
In this section, we study some natural bundles on the generalized projective product spaces considered in  \Cref{sec:com_proj_prod_sps}. Then we compute their Chern characteristic classes. 

\subsection{Stable tangent bundle on some manifolds in \ref{ex:proj_prod_sp}}
We follow the notation used in  \Cref{ex:proj_prod_sp}. 
The following result identifies the stable isomorphism type of the tangent bundle of $C_{(\overline{n}, \overline{p})(\ell),\overline{m}(k)}$. 
\begin{theorem}\label{stab-bnd1}
The tangent bundle $\displaystyle \mathcal{T}(C_{(\overline{n}, \overline{p})(\ell),\overline{m}(k)})$  is stably isomorphic to the bundle \[\bigg(\sum_{i=1}^k (m_i+1) \oplus \sum_{j=1}^{\ell}(n_j-p_j+2)\bigg)\eta_{\ell+1} \oplus\sum_{j=1}^\ell (p_j-1) \varepsilon_{\CC},\]
where $\eta_{\ell+1}$ is defined in \eqref{eq:can_line_bundle}.
\end{theorem}
\begin{proof}
 The tangent bundle on the product $ S(\overline{m}(k)) \times  S(\overline{n}(\ell))$ is given by
 $$ \{( \bar{\bf x}, \bar{\bf y}, \bar{\bf u}, \bar{\bf v}) \in  \prod_{i=1}^{k} S^{2m_i+1}\times \prod_{j=1}^{\ell} S^{2n_j+1} \times \prod_{i=1}^{k}\CC^{m_i+1} \times \prod_{j=1}^{\ell} \CC^{n_j+1}~{\big |}~{\bf u}_i \perp {\bf x}_i~ \& ~ {\bf v}_j \perp {\bf y}_j\} $$
 where $\bar{\bf x}=({\bf x}_1, \ldots, {\bf x_k})$, $\bar{\bf y} = ({\bf y}_1, \ldots, {\bf y}_{\ell}),$ $\bar{\bf u}=({\bf u}_1, \ldots, {\bf u_k})$, $\bar{\bf v} = ({\bf v}_1, \ldots, {\bf v}_{\ell}),$ $1 \leq i \leq k$ and $1 \leq j \leq \ell$, and ${\bf u}_i \perp {\bf x}_i$ denotes that ${\bf u}_i, {\bf x}_i$ are orthogonal with respect to the real inner product.  We consider the equivalence  
  relation $\sim$ on the tangent space $\mathcal{T}(S(\overline{m}(k)) \times  S(\overline{n}(\ell)) )$ defined by $( \bar{\bf x}, \bar{\bf y}, \bar{\bf u}, \bar{\bf v})  \sim (\tau(\bar{\bf x}), \sigma(\bar{\bf y}), \tau (\bar{\bf u}), \sigma(\bar{\bf v}))$, where 
  $$\tau(\bar{\bf u})=(\tau_1,\ldots,\tau_k)(\bar{\bf u})=(\tau_1({\bf u}_1),\ldots, \tau_k({\bf u}_k)),$$ $$\sigma(\bar{\bf v})=(\sigma_1,\ldots,\sigma_\ell)(\bar{\bf v})=(\sigma_1({\bf v}_1),\ldots,\sigma_\ell({\bf v}_\ell))$$ and the actions $\tau$'s, $\sigma_j$'s are defined in Example \ref{eq:multiple_reflection}. 
So, the tangent space on $C_{(\overline{n}, \overline{p})(\ell),\overline{m}(k)}$ is given by the equivalence classes $$\displaystyle\{[ \bar{\bf x}, \bar{\bf y}, \bar{\bf u}, \bar{\bf v}]  \colon ( \bar{\bf x}, \bar{\bf y}, \bar{\bf u}, \bar{\bf v}) \in \mathcal{T}( S(\overline{m}(k)) \times  S(\overline{n}(\ell)))\}.$$ Therefore, we have the following isomorphism of vector bundles 
$$ \mathcal{T}(\displaystyle \prod_{i=1}^k S^{2m_i+1} \times \displaystyle \prod_{i=1}^{\ell} S^{2n_j+1}) \oplus (k+\ell)\varepsilon_{\RR} \approx 
\sum_{i=1}^k (m_i+1) \varepsilon_{\CC} \oplus \sum_{j=1}^\ell (p_j-1) \varepsilon_{\CC} \oplus \sum_{j=1}^{\ell}(n_j-p_j+2)\varepsilon_{\CC} $$
$$\approx \sum_{i=1}^k (m_i+1) \varepsilon_{\CC} \oplus \sum_{j=1}^\ell (n_j+1) \varepsilon_{\CC}, $$
 where $\varepsilon_{\FF}$ represents the trivial line bundle on $ S(\overline{m}(k)) \times  S(\overline{n}(\ell))$ with fiber $\FF$. 
 Now consider the natural $S^1$-actions on the both sides where $S^1$ acts on the space $ \mathcal{T}( S(\overline{m}(k)) \times  S(\overline{n}(\ell)))$ by $ (\bar{\bf x}, \bar{\bf y}, \bar{\bf u}, \bar{\bf v}) = (\tau\cdot\bar{\bf x}, \sigma(\bar{\bf y}), \tau\cdot\bar{\bf u}, \sigma(\bar{\bf v}))$, on $(k+\ell)\varepsilon_{\RR}$ trivially, on each $(m_i+1)\varepsilon_{\CC}$ via complex multiplication, on each $(p_j-1)\varepsilon_{\CC}$ trivially, and on each $(n_j-p_j+2) \varepsilon_{\CC}$ via complex multiplication. This implies that the following bundle map is $S^1$-equivariant. 
\begin{center}
\begin{tikzcd}
\mathcal{T}( \displaystyle \prod_{i=1}^k S^{2m_i+1} \times \displaystyle \prod_{i=1}^{\ell} S^{2n_j+1}) \oplus (k+\ell)\varepsilon_{\RR} \arrow{d}{} \arrow{r}{\cong}
&  \displaystyle \sum_{i=1}^k (m_i+1) \varepsilon_{\CC} \oplus \sum_{j=1}^\ell (n_j+1) \varepsilon_{\CC}  \arrow{d}{} \\
S(\overline{m}(k)) \times  S(\overline{n}(\ell)) \arrow{r}{=}
&  S(\overline{m}(k)) \times  S(\overline{n}(\ell))
\end{tikzcd}
\end{center}
This is induced from the $S^1$-action on  $ S(\overline{m}(k)) \times  S(\overline{n}(\ell))$. So  we get that the tangent bundle $\displaystyle \mathcal{T}(C_{(\overline{n}, \overline{p})(\ell),\overline{m}(k)})$ is stably isomorphic to the bundle 
$$\big{(}\sum_{i=1}^k (m_i+1) \oplus \sum_{j=1}^{\ell}(n_j-p_j+2)\big{)}\eta_{\ell+1} \oplus\sum_{j=1}^\ell (p_j-1) \varepsilon_{\CC}. $$ 
Here the line bundle $\eta_{\ell+1}$ is defined on  $C_{(\overline{n}, \overline{p})(\ell),\overline{m}(k)}$ for $\ell \geq 0$, see Subsection \ref{subsec:proj_prod_sp}.   
\end{proof}

We note that the stable tangent bundle of the complex projective product space $C_{\overline{m}(k)}$ is discussed in \cite[Proposition 3]{Gonzalez-Velasco}.

\begin{remark}\label{chern_class_1}
From the stable tangent bundle isomorphism, one can compute the total Chern class of  $C_{(\overline{n}, \overline{p})(\ell),\overline{m}(k)})$ where the Chern class of $\eta_\ell$ is given by $c(\eta_\ell) = (1+c_1(\eta_\ell))$.
\end{remark}

\subsection{Stable tangent bundle of some manifolds in \Cref{ex:toric}}\label{sec:toric_bundle}
In this subsection, we construct several complex line bundles on the generalized projective product spaces defined in \Cref{ex:toric}.

Let $X$ be a toric manifold over a simple polytope $Q$ which has $\mu$ many facets. Consider the line bundle $ X \times S(\overline{m}(k)) \times \mathbb{C}  \to X \times S(\overline{m}(k)) $ and the $S^1$-action on the total space defined by $$(y, {\bf x}_1, \ldots, {\bf x}_k, t) \mapsto (\rho(y), \tau_1{\bf x}_1, \ldots, \tau_k{\bf x}_k, tz)$$ where $\rho$ is as in  \Cref{ex:toric}. So, the bundle map is $S^1$-equivariant, and it induces the following line bundle $$\eta \colon   \big{(} X \times S(\overline{m}(k)) \times \mathbb{C} \big{)}/{S^1}  \to  C(X,S(\overline{m}(k))). $$ 

The fixed point set of the action $\rho$ on $X$ is non-empty by definition.  Let $y_0$ be a fixed point of this $S^1$-action on $X$. Then $ {y_0 \times S(\overline{m}(k)) \subseteq X \times S(\overline{m}(k))}$. The inclusion $\iota \colon  C_{\overline{m}(k)} \subseteq C(X, S(\overline{m}(k)))$ then gives the following commutative diagram.  
\begin{equation}\label{eq:cano_line_bund}
\begin{tikzcd}
\big{(} \{y_0\} \times S(\overline{m}(k)) \times \mathbb{C} \big{)}/{S^1} \arrow{d}{\eta_1} \arrow{r}{\overline{\iota}}
& \big{(} X \times S(\overline{m}(k)) \times \mathbb{C} \big{)}/{S^1} \arrow{d}{\eta} \\
C_{\overline{m}(k)} \arrow{r}{\iota}
&  C(X, S(\overline{m}(k))).
\end{tikzcd}
\end{equation}
 Then $\iota^*(\eta)=\eta_1$ where $\eta_1$ is defined in \eqref{eq:can_line_bundle} for $\ell=1$. By naturality, the total Chern class of $\eta$ is given by $c(\eta)= 1+ c_1(\eta)$, where 
\begin{equation}\label{eq:sw_cla_eta1}
c_1(\eta_1) =\iota^*c_1(\eta)). 
\end{equation}

Recall from \Cref{moment-angle} that  $Z_Q$ is the moment angle manifold corresponding to $X$ and the group $T^{\mu-n} = \ker({\rm exp}\Lambda)$ in \eqref{eq:torus_exact} is a  subtorus of $T^{\mu}$ such that $Z_Q/T^{\mu-n} \cong X$ and $T^n = T^{\mu}/T^{\mu -n} $ action on $X$ is locally standard. We denote the natural action  $$T^{\mu-n} \times Z_Q \to Z_Q$$ by $\nu$.  Let $\pi_i \colon T^{\mu} \to S^1$ be the projection onto the $i$th factor. The torus $T^{\mu}$ acts on $\CC$ via this projection by complex multiplication.  So $T^{\mu-n}$ acts on $\CC$ via the composition $T^{\mu-n} \hookrightarrow T^{\mu} \xrightarrow{\pi_i} S^1$. We denote this one dimensional representation of $T^{\mu-n}$ by $\CC_i$ and the associated action by $\rho_i$ for $i=1, \ldots, \mu$.  

Recall the $S^1$-action on a toric manifold $X =Z_Q/T^{\mu-n}$ from \Cref{subsec-qt}. Then there is a circle subgroup $S^1_1 \subseteq T^{\mu}$ such that $\exp \Lambda(S^1_1)=S^1$. Note that $T^{\mu-n} \cap S^1_1=Id \in T^\mu$. Also, $T^{\mu-n} \times S^1_1 \subseteq T^\mu$ acts on $Z_Q$ naturally. Now define an $T^{\mu-n} \times S^1_1$ action on $Z_Q \times S(\overline{m}(k)) \times \CC_i$ defined by 
\begin{equation}\label{eq:iden_bundle}
((t, s), ({\bf x}, {\bf y}, z)) \mapsto ((t,s){\bf y}, s{\bf x}, s{\rho_i(t, z)}).
\end{equation} 
Then the identification space gives a complex line bundle $$\zeta_i \colon Z_Q \times S(\overline{m}(k)) \times \CC_i)/T^{\mu-n} \to C(X, S(\overline{m}(k)))$$ on $C(X, S(\overline{m}(k)))$, denoted by $\zeta_i$ for $i=1, \ldots, \mu$. 

Next, we recall some canonical complex line bundles on $X$ following \cite{BP}. The trivial complex line bundle $Z_Q \times \CC_i \to Z_Q$ is equivariant with respect to the action of $T^{\mu-n}$ for $i=1, \ldots, \mu$. Then, this bundle induces a  complex line bundle $$L_i \colon Z_Q \times_{T^{\mu-n}} \CC_i \to X .$$
The first Chern class of the bundle $L_i$ is given by $c_1(L_i) = u_i$, where $u_i \in H^2(X; \ZZ)$ is the Poincar{\'e} dual to the characteristic submanifold $\mathfrak{q}^{-1}(F_i)$, see  \cite[Section 6]{DJ}, and $F_i$ is the $i$th facet of $Q$. Moreover, if we fix an orientation on $Q$ and $T^{n}$, then \cite[Theorem 5.33]{BP} gives 
 \begin{equation}\label{stab-cplx}
 T(X) \oplus \RR^{2(\mu-n)} \cong L_1 \oplus \cdots \oplus L_\mu.
 \end{equation}

Let $a=(e_{1_1}, \ldots, e_{1_k}) \in S(m_1, \ldots, m_k)~~$ where $e_{1_i}$ is the first vector in the standard basis of $\RR^{m_i+1}$. Let $aS^1 \in S(m_1, \ldots, m_k)$ be the orbit of $a$ for the diagonal $S^1$-action on $S(m_1, \ldots, m_k)$. So we get the inclusion $$\phi \colon X = C(X, aS^1) \subset C(X, S(\overline{m}(k))).$$ Then the pull-back of $\zeta_i$ under $\phi$ is $L_i$  for $i=1, \ldots, \mu$. By naturality $\phi^* c_1(\zeta_i) = u_i$ for  $i=1, \ldots, \mu$.

Using the above discussion and the proof of Theorem \ref{stab-bnd1}, one gets the following. 

\begin{theorem}
The bundle $\mathcal{T}(C(X, S(\overline{m}(k)))) \oplus (k +2\mu -2n) \varepsilon$ is isomorphic to $$\displaystyle \sum_1^k(m_i+1) \eta \oplus 
 \zeta_1 \oplus \cdots \oplus \zeta_\mu.$$
\end{theorem}

\begin{cor}
If $m_1, \ldots, m_k$ are greater than 1, then the total Chern class of  $C(X, S(\overline{m}(k)))$ is given by $$c (C(X, S(\overline{m}(k))))= (1+v)^{m} \prod_{j=1}^{\mu}(1+ u_j),$$ where $m = \displaystyle \sum_1^k m_i +k$. 
\end{cor}

\section{Cat and TC of generalized complex projective product spaces}\label{sec-LSCat-TC}

In this section, we discuss the product inequality for strongly equivaraint TC. Then we exhibit the closeness of the lower and upper bounds for the LS-category and TC of several classes of generalized complex projective product spaces. In many cases, we compute the exact value of the LS-category.

\begin{lemma}[{\cite[Lemma 6.1]{F-G-L-O}}]
\label{lem: equivalent criterion n+1 cover}
Let $\mathcal{V}=\{V_1,\dots,V_{k+n+1}\}$
be a cover of $X$. Then $\mathcal{V}$ is an $(n+1)$-cover if and only if each $x\in X$ is contained in at least $k+1$ sets of $\mathcal{V}$.
\end{lemma}

\begin{theorem}[{\cite[Theorem 2.5]{dranishnikov2009lusternik}}]
\label{thm: extend cover to k+1 cover}
Let $\mathcal{U}=\{U_1,\dots,U_{k+1}\}$
be an open cover of a normal topological space $X$.
Then, for any $\ell\geq k$, there is an open $(k+1)$-cover
$\{U_1,\dots,U_{\ell+1}\}$
of $X$, extending $\mathcal{U}$, such that, for $n>k+1$, $U_n$ is a disjoint union of open sets, each of which is contained in one of the sets $U_j$, where $1\leq j\leq k+1$.
\end{theorem}

For a $G$-space $X$, the \emph{strongly equivariant topological complexity} was introduced by Dranishnikov in \cite{strongeqtc}. It is denoted by $\TC_G^*(X)$ and is defined as the smallest positive integer $k$ such that $X\times X$ admits a cover by $(G\times G)$-invariant open sets ${U_1,\ldots,U_k}$ for which, for each $i$, there exists a $G$-section
$s_i\colon U_i\to PX$ of the fibration
$\pi_X\colon PX\to X\times X$.

We next establish the product inequality for strongly equivariant topological complexity.

\begin{theorem}
Let $X$ and $Y$ be connected normal $G_1$ and $G_2$-spaces, respectively. Then
\[
\TC_{G_1\times G_2}^*(X\times Y)
\leq
\TC_{G_1}^*(X)+\TC_{G_2}^*(Y)-1.
\]
\end{theorem}

\begin{proof}
Let $\TC_{G_1}^*(X)=k$
and
$\TC_{G_2}^*(Y)=n.$ So, there exist $G_1\times G_1$-invariant and $G_2\times G_2$-invariant open covers $\mathcal{U}=\{U_1,\dots,U_k\}$
of $X\times X$, and $\mathcal{V}=\{V_1,\dots,V_n\}$
of $Y\times Y$,
such that each $U_i$ (respectively, $V_j$) admits a $G_1$-equivariant (respectively, $G_2$-equivariant) section of the free path fibration
$\pi_X\colon PX\longrightarrow X\times X$, and
${\pi_Y\colon PY\longrightarrow Y\times Y.}$
By Theorem~\ref{thm: extend cover to k+1 cover}, extend $\mathcal{U}$ to a $(k+1)$-cover $\mathcal{U}'=\{U_1,\dots,U_{k+n-1}\}$,
and $\mathcal{V}$ to an $(n+1)$-cover
$\mathcal{V}'=\{V_1,\dots,V_{k+n-1}\}$,
such that for $i>k$ (respectively, $i>n$), $U_i$ (respectively, $V_i$) is a disjoint union of open sets contained in some $U_j$ (respectively, $V_j$), where $1\leq j\leq k$ (respectively, $1\leq j\leq n$).

Note that $\Phi\colon
(X\times X)\times(Y\times Y)
\longrightarrow
(X\times Y)\times(X\times Y)$
defined by $\Phi((x_1,x_2,y_1,y_2))
=
((x_1,y_1),(x_2,y_2))$ 
is a $(G_1\times G_2)\times(G_1\times G_2)$-equivariant homeomorphism.
For $1\leq i\leq k+n-1$, define
\[
W_i:=\Phi(U_i\times V_i)
\subset
(X\times Y)\times(X\times Y),
\]
and set
$\mathcal{W}=\{W_1,\dots,W_{k+n-1}\}.$
Note that each $W_i$ is $(G_1\times G_2)\times(G_1\times G_2)$-invariant.

We claim that $\mathcal{W}$ covers $(X\times Y)\times(X\times Y)$.
If not, then there exists
\[
((x_1,y_1),(x_2,y_2))
\notin W_i
\qquad\text{for all }i.
\]
This means that, for each $i$, $(x_1,x_2)\notin U_i$ or
$(y_1,y_2)\notin V_i$.
Since $\mathcal{U}'$ is a $(k+1)$-cover, by Lemma~\ref{lem: equivalent criterion n+1 cover}, $(x_1,x_2)$ belongs to at least $n+1$ sets among $\{U_i\}$. Without loss of generality, assume
$(x_1,x_2)\in U_1,\dots,U_{n+1}$.
Then $(y_1,y_2)\notin V_1,\dots,V_{n+1}$.

Hence $(y_1,y_2)$ can belong only to $V_{n+2},\dots,V_{n+k-1}$.
Thus, $(y_1,y_2)$ belongs to at most $k-2$ sets of $\mathcal{V}'$, which contradicts the fact that $\mathcal{V}'$ is an $(n+1)$-cover. Hence, $\mathcal{W}$ is a cover.

Next, we construct $(G_1\times G_2)$-equivariant sections over each $W_i$. Using the natural homeomorphism
$P(X\times Y)\cong_{G_1\times G_2}PX\times PY$,
the free path fibration satisfies $\pi_{X\times Y}=\pi_X\times\pi_Y$.

Let $s_i^X\colon U_i\longrightarrow PX$, and $s_i^Y\colon V_i\longrightarrow PY$
be $G_1$-equivariant and $G_2$-equivariant sections of $\pi_X$ and $\pi_Y$, respectively. Define
\[
s_i\colon
W_i=U_i\times V_i
\longrightarrow
PX\times PY
\cong P(X\times Y)
\quad 
\text{by}
\quad
s_i(u,v)
=
\big(s_i^X(u),s_i^Y(v)\big).
\]
Then $s_i$ is $(G_1\times G_2)$-equivariant and satisfies
$\pi_{X\times Y}\circ s_i
=
\operatorname{id}_{W_i}.$
Thus, $\mathcal{W}$ is a $(G_1\times G_2)\times(G_1\times G_2)$-invariant open cover with $k+n-1$ sets admitting $(G_1\times G_2)$-equivariant sections. This gives the desired inequality.
\end{proof}

\begin{prop}[{\cite[Theorem 2.6]{Naskar}}]\label{thm:cat}
Let $F\hookrightarrow{} E\stackrel{p}{\longrightarrow} B$ be a fibre bundle, where  $F$, $E$ and $B$ are path-connected, Hausdorﬀ, second countable topological spaces. Suppose that $E$ is completely normal. 
Let $\{U_1,\dots,U_{m} \}$ be a categorical cover of $B$
such that $\phi_i \colon \overline{p^{-1}(U_i)}\to \overline{U_i}\times F$ is a homeomorphism for $1\leq i\leq m$. Moreover, there is a categorical cover  $\{V_1,\dots,V_{n}\}$ of $F$ such that $(U_{i'}\cap U_i)\times V_j$ are invariant under $\phi_{i'}\phi_{i}^{-1}$ for all $\{i,i'\}\subseteq \{1,\dots,m\}$ and $j\in\{1,\dots,n\}$. Then $\mathrm{cat}(E)\leq \mathrm{cat}(F)+\mathrm{cat}(B)-1$.
\end{prop}

Suppose that $\left<\tau \right>$ denotes a $S^1$-action on $M$ and $\left< \sigma \right>$ denotes a free $S^1$-action on $N$. The following result is a consequence of \Cref{thm:cat}.
\begin{prop}\label{prop: catCMN}
Let $C(M, N)$ be a generalized complex projective product space. Let $\{V_1,\dots, V_q\}$ be a $\left<\tau \right>$-invariant categorical cover of $M$. Then \[\ct(C(M,N)) \leq q +\ct(N/\left<\sigma \right>)-1.\]
In particular, we have \[\ct(C(M,N)) \leq \ct_{S^1}(M)+ \ct(N/\left<\sigma \right>) -1.\]
\end{prop}
\begin{proof}
Let $\ct(N/\left<\sigma \right>)=r$ and $\{U_1,\dots,U_r\}$ be a categorical cover of $N/\left<\sigma \right>$. The orbit map ${\pi \colon N \to N/\left<\sigma \right>}$  is a principal $S^1$-bundle. Therefore, for $1\leq i\leq r$, we have the following. \[p^{-1}(U_i)=\frac{M\times \pi^{-1}(U_i) }{(x,y)\sim (\tau(x),\sigma(y))}\cong U_i\times M.\] 
Let $\phi_i \colon p^{-1}(U_i)\to U_i\times M$ be this homeomorphism. 
So, it is a local trivialization of ${M \hookrightarrow C(M, N) \stackrel{\mathfrak{p}}\longrightarrow N/\left<\sigma \right>}$ for $i=1, \ldots, r$. Now $$\phi_{i}^{-1}(([y],x))=[(x,y)]\in M\times \pi^{-1}(U_i)/\sim$$ if $[y] \in U_i$. 
The structure group for the bundle ${M \hookrightarrow C(M, N) \xrightarrow{p} N/\left<\sigma \right>}$ is $S^1$. Therefore, either
\[\phi_{j}\circ\phi_{i}^{-1}(([y],x))= 
    ([y],x) \mbox{ or } \phi_{j}\circ\phi_{i}^{-1}(([y],x)) = ([y],\tau(x)).\]
Thus, $(U_i\cap U_j)\times V_k$ is invariant under  $\phi_{j}\circ\phi_{i}^{-1}$ for all $1\leq i,j\leq r$ and $1\leq k\leq q$. 
Then, the proposition follows from \Cref{thm:cat}, since $M$, $N$ and $C(M,N)$ are manifolds.
\end{proof}

\begin{prop}\label{prop:TC of gcpps}
Let $C(M, N)$ be a generalized complex projective product space. Let $\{V_1',\dots,V_q'\}$ be a $(\left<\tau \right>\times \left< \tau \right>)$-invariant motion planning cover of $M$.
Then \[\TC(C(M,N))\leq q + \ct(N/\left<\sigma \right> \times N/\left<\sigma \right>)-1.\]
In particular, we have \[\TC(C(M,N))\leq \TC_{S^1}^*(M) + \ct(N/\left<\sigma \right> \times N/\left<\sigma \right>)-1.\]
\end{prop}

\begin{proof} 
Let  $H_i \colon U_i' \times [0,1] \to N/\left<\sigma \right> \times N/\left<\sigma \right>$ be the null-homotopy for some open subsets $U_1', \ldots, U_m'$ of $N/\left<\sigma \right> \times N/\left<\sigma \right>$. So, the image of $H_i$ is a contractible subset of $N/\left<\sigma \right> \times N/\left<\sigma \right>$, consequently, each $U_i'$ is a subset of a contractible subset of $N/\left<\sigma \right> \times N/\left<\sigma \right>$ and $U_1', \ldots, U_m'$ is a motion planning cover of $N/\left<\sigma \right> \times N/\left<\sigma \right>$. Therefore, $U_i'$'s satisfy the first condition of \cite[Theorem 2.2]{DSTCgpps}, and for $1\leq i\leq m$, we have the following. 
\[(p \times p)^{-1}(U_i)=\frac{M \times M \times (\pi \times \pi)^{-1}(U_i') }{(x_1,y_1,x_2,y_2)\sim_p (g\cdot (x_1, y_1),h\cdot( x_2,y_2))} \cong (M \times M) \times U_i',\]
where $(g, h) \in \left < \tau\right>\times \left < \sigma\right>$ and $(\pi \times \pi)^{-1}(U_i') \subseteq N \times N$.  
We denote this homeomorphism by $h_i \colon (p \times p)^{-1}(U_i')\to (M \times M) \times U_i'$. So, it is a local trivialization of
\begin{equation}\label{eq: prod bundle}
M \times M \hookrightarrow{} C(M, N) \times C(M, N) \stackrel{\mathfrak{p} \times \mathfrak{p}}\longrightarrow N/\left<\sigma \right> \times N/\left<\sigma \right>
\end{equation}
for $i=1, \ldots, m$. Now, $$h_{i}^{-1}(([y_1, y_2], x_1, x_2))=[(x_1, x_2, y_1, y_2)]\in M \times M \times (\pi \times \pi)^{-1}(U_i')/\sim_p$$ if $[y_1, y_2] \in U_i'$. Since the structure group for the bundle \eqref{eq: prod bundle} is $(S^1)^2 \cong \left < \tau \right>^2$, then
\[h_{j}\circ h_{i}^{-1}(( x_1, x_2, [y_1, y_2])) = (\tau_1(x_1), \tau_2(x_2), [y_1, y_2])\] for some $\tau_1, \tau_2 \in \left< \tau \right>$. Thus, the set $ V_k' \times (U_i'\cap U_j')$ is invariant under  $h_{j}\circ h_{i}^{-1}$ for all $1\leq i, j\leq m$ and $1\leq k\leq q$. Note that $M, N$ are regular since they are manifolds. 
Thus $U_i'$'s and $V_\ell'$'s satisfy the hypotheses of \cite[Theorem 2.2]{DSTCgpps}. Hence, the conclusion follows.
\end{proof}

\begin{example}\label{thm: eqctoddsph}
 Suppose that the $S^1$-action on $S^{2n_j+1}$ is given as in \Cref{ex:proj_prod_sp}. If $p_j\geq 1$, then the fixed point set $(S^{2n_j+1})^{S^1}=S^{2p_j-1}$. Therefore, $S^{2n_j+1}$ is $S^1$-connected. Then using \cite[Section 2]{angel2018equivariant}, one can get 
  \[2=\ct(S^{2n_j+1})\leq \ct_{S^1}(S^{2n_j+1}).\] 
Now one can note that the open subsets $U_1=S^{2n_j+1}\setminus \{(1,0,\dots,0)\}$ and $U_2=S^{2n_j+1}\setminus \{(-1,0,\dots,0)\}$ forms an $S^1$-equivariant categorical cover of $S^{2n_j+1}$.
Thus, we get that $\ct_{S^1}(S^{2n_j+1})=2$.

If $p_j=0$,
then the $S^1$-action is free. Then using \cite[Theorem 1.15]{eqlscategory} we have $\ct_{S^1}(S^{2n_j+1})=\ct(\CC P^{n_j})$. Thus,
 \[\ct_{S^1}(S^{2n_j+1})=\begin{cases}
     2 &\ \ if \ \ p_j\geq 1,\\
     n_j+1 & \ \ if \ \ p_j=0.
 \end{cases}\]
\end{example}

\begin{prop}\label{prop: ctCniN}
Let $n_1\leq \dots \leq n_{\ell}$ and  $N$ be a smooth manifold having free circle action denoted by $\tau$. Suppose $p_j\geq 1$ for $1\leq j\leq \ell$. Then
\begin{equation}\label{eq: ctcMN}
  \cl(N/\left<\tau\right >)+\ell+1\leq \ct(C((\overline{n}, \overline{p})(\ell),N))\leq \ct(N/\left<\tau\right >)+\ell .  
\end{equation}
\end{prop}
\begin{proof}
It follows from the product inequality \cite[Proposition 2.10]{BaySarkarheqtc} of equivariant category that, 
$\ct_{S^1}(\prod_{i=1}^{\ell} S^{2n_i+1})\leq \ell +1$. Therefore, using \Cref{prop: catCMN}, we get the left inequality of \eqref{eq: ctcMN}.
The right inequality of \eqref{eq: ctcMN} follows from the cup length calculations.
\end{proof}

Let $C_{(\overline{n}, \overline{p})(\ell),\overline{m}(k)}$ be a generalized complex projective product space as defined in \Cref{ex:proj_prod_sp}. We compute bounds on their LS category and topological complexity.
The upper bound on the LS-category of complex projective product spaces was computed in \cite[Section 4]{Gonzalez-Velasco}.
We now compute the exact value of the LS-category of complex projective product spaces. 

\begin{theorem}\label{thm: ct gcpps}
   Let $m_1\leq \dots\leq m_k$. Then $\ct(C_{\overline{m}(k)})=m_1+k.$
\end{theorem}
\begin{proof}
We can use the cohomology description given in \cite[Theorem 1]{Gonzalez-Velasco} to conclude that $\cl(C_{\overline{m}(k)})=m_1+k-1$.
Therefore,  $m_1+k-1\leq \ct(C_{\overline{m}(k)})$. Now it follows from \Cref{prop: catCMN} that 
\[\ct(C_{\overline{m}(k)})\leq \ct\big(\CC P^{m_1}\big)+\ct_{S^1}\bigg(\prod_{i=2}^kS^{2m_i+1}\bigg)-1.\]
Therefore, using \Cref{thm: eqctoddsph} and the product inequality \cite[Prop 2.10]{BaySarkarheqtc}, we get that $\ct(C_{\overline{m}(k)})\leq m_1+1+k-1=m_1+k$.
This completes the proof.
\end{proof}

\begin{theorem}\label{thm-ckl}
Let $C_{(\overline{n}, \overline{p})(\ell),\overline{m}(k)}$ be a generalized complex projective product space with $p_j\geq 1$ for $1\leq j\leq \ell$. Then
$\ct(C_{(\overline{n}, \overline{p})(\ell),\overline{m}(k)})=m_1+k+\ell.$
\end{theorem}
\begin{proof}
From \Cref{prop: ctCniN} and \Cref{thm: ct gcpps} we have, \[\ct(C_{(\overline{n}, \overline{p})(\ell),\overline{m}(k)})\leq m_1+k+\ell.\] The lower bound follows from the cup-length calculations of cohomology ring described in \Cref{prop_cohom_gen_proj_prod}.
 \end{proof}

\begin{theorem}
Let $m_1\leq \dots\leq m_k$.   Then \[2m_1+k\leq \TC(C_{\overline{m}(k)})\leq 2m_1+2k-1.\]    
\end{theorem}

\begin{proof}
Note that the $\zl(C_{\overline{m}(k)})=2m_1+k-1$.
Therefore,  $2m_1+k\leq \TC(C_{\overline{m}(k)})$. 
The inequality $\TC(C_{\overline{m}(k)})\leq 2m_1+2k-1$ follows using \cite[Theorem 5]{Far} and Theorem \ref{thm: ct gcpps}.
\end{proof}

\begin{theorem}
Let $C_{(\overline{n}, \overline{p})(\ell),\overline{m}(k)}$ be  a generalized complex projective product space. Then
\begin{equation}\label{eq: tc cpps1}
2m_1+k+\ell\leq \TC(C_{(\overline{n}, \overline{p})(\ell),\overline{m}(k)})\leq  2(m_1+k+\ell)-1.    
\end{equation}
\end{theorem}
\begin{proof}
It follows from \Cref{prop_cohom_gen_proj_prod} that $\zl(C_{(\overline{n}, \overline{p})(\ell),\overline{m}(k)})=2m_1+k+\ell-1$. Thus, we get the left inequality of \eqref{eq: tc cpps1} from \cite[Theorem 7]{Far}. 

Now, from the product inequality of LS-category and \Cref{thm-ckl}, we get the right inequality of \eqref{eq: tc cpps1}.
\end{proof}

\begin{theorem}\label{thm-grass-gen-dold}
Let $ \G(d,n)$ be a Grassmann manifold. Then
\[\ct(C(Gr_d(\CC^n), S(\overline{m}(k))))= d(n-d)+m_1+k.\]
\end{theorem}
\begin{proof}
The inequality $d(n-d)+m_1+k \leq \ct(C(Gr_d(\CC^n), S(\overline{m}(k)))$ follows from the cup-length calculations of the cohomology ring described in \Cref{thm:toric_mod2}. For calculating the cup-length of the cohomology of $ \G(d,n)$, one can see the paragraph before \cite[Theorem 5.1]{DSTCgpps}. 
Note that from Proposition~\ref{prop: catCMN} we have 
\[
\ct(C(Gr_d(\CC^n), S(\overline{m}(k))))\leq q+ \ct(C_{\overline{m}(k)}))-1,
\]
where $q$ is the number of $S^1$-invariant categorical subsets covering $Gr_d(\CC^n)$. Observe that the $\ZZ_2$-invariant categorical cover constructed in the proof of \cite[Theorem 5.1]{DSTCgpps} is indeed a $S^1$-invariant categorical cover. Then using Theorem~\ref{thm: ct gcpps}, we obtain 
\[
\ct(C(Gr_d(\CC^n), S(\overline{m}(k))))\leq d(n-d)+1+ m_1+k-1.
\]
This completes the proof.
\end{proof}

\begin{prop}
 Let $ \G(d,n)$ be a Grassmann manifold. Then
\[2d(n-d)+2m_1+k+1\leq \TC(C(Gr_d(\CC^n), S(\overline{m}(k))))\leq 2(d(n-d)+ m_1+k).\]   
\end{prop}
\begin{proof}
 From \Cref{thm:toric_mod2}, it follows that 
 $$
\zl(C(Gr_d(\CC^n), S(\overline{m}(k))))=\zl(Gr_d(\CC^n))+\zl(C_{\overline{m}(k)}).
$$ 
 Moreover, $\zl(Gr_d(\CC^n))=2d(n-d)$ and $\zl(C_{\overline{m}(k)})=2m_1+k$.
 This gives us the required lower bound. The upper bound follows from \Cref{thm-grass-gen-dold} and using the inequality in \cite[Theorem 5]{Far}.
\end{proof}

\begin{theorem}Let $M^{2r}$ be a quasitoric manifold of dimension $2r$. Then
\[\ct(C(M^{2r}, S(\overline{m}(k))))= r+m_1+k.\]   
\end{theorem}
\begin{proof}
The arguments are similar to the proof of \Cref{thm-grass-gen-dold}, using \Cref{thm:toric_mod2}. So, we omit the details. 
\end{proof}

\begin{prop}
Let $M^{2r}$ be a quasitoric manifold of dimension $2r$. Then
\[2r+2m_1+k+1 \leq\TC(C(M^{2r}, S(\overline{m}(k))))\leq 2(r+m_1+k).\] 
In particular, if $k=1$, then we have $\TC(C(M^{2r}, S(\overline{m}(k))))=2(r+m_1+k).$
\end{prop}
\begin{proof}
From \Cref{thm:toric_mod2}, it follows that 
 $$
\zl(C(M^{2r}, S(\overline{m}(k))))=\zl(M^{2r})+\zl(C_{\overline{m}(k)}).
$$ 
 Moreover, $\zl(M^{2r})=2r$ (see \cite[Theorem 2.8 and Corollary 2.10]{B-D-S}) and $\zl(C_{\overline{m}(k)})=2m_1+k$.
 This gives us the required lower bound. The upper bound follows from \Cref{thm-grass-gen-dold} and using the inequality in \cite[Theorem 5]{Far}.
\end{proof}

\noindent{\bf Acknowledgement:} The first author gratefully acknowledges the support of DST-INSPIRE Faculty Fellowship (Faculty Registration No. IFA24-MA218), as well as Industrial Consultancy and Sponsored Research (IC\&SR), Indian Institute of Technology Madras for the New Faculty Initiation Grant (RF25261395MANFIG009294). The second author thanks ANRF India for the research grants CRG/2023/000239 and ANRF/ARGM/2025/002640/MTR.


\begin{thebibliography}{10}

\bibitem{angel2018equivariant}
A.~Angel and H.~Colman.
\newblock Equivariant topological complexities.
\newblock {\em Topological complexity and related topics, Contemp. Math},
  702:1--15, 2018.

\bibitem{Aud}
M.~Audin.
\newblock {\em Torus actions on symplectic manifolds}, volume~93 of {\em
  Progress in Mathematics}.
\newblock Birkh\"auser Verlag, Basel, revised edition, 2004.

\bibitem{B-D-S}
M.~Bayeh, N.~Daundkar, and S.~Sarkar.
\newblock An exploration of {LS} category and topological complexity of {D}old
  manifolds of toric type.
\newblock {\em J. Topol. Anal.}, 18(6):1875--1893, 2026.

\bibitem{BaySarkarheqtc}
M.~Bayeh and S.~Sarkar.
\newblock Higher equivariant and invariant topological complexities.
\newblock {\em J. Homotopy Relat. Struct.}, 15(3-4):397--416, 2020.

\bibitem{BP}
V.~M. Buchstaber and T.~E. Panov.
\newblock {\em Torus actions and their applications in topology and
  combinatorics}, volume~24 of {\em University Lecture Series}.
\newblock American Mathematical Society, Providence, RI, 2002.

\bibitem{Pan}
V.~M. Buchstaber and T.~E. Panov.
\newblock {\em Torus actions and their applications in topology and
  combinatorics}, volume~24 of {\em University Lecture Series}.
\newblock American Mathematical Society, Providence, RI, 2002.

\bibitem{DSTCgpps}
N.~Daundkar and S.~Sarkar.
\newblock L{S}-category and topological complexity of several families of fibre
  bundles.
\newblock {\em Homology Homotopy Appl.}, 26(2):273--295, 2024.

\bibitem{DJ}
M.~W. Davis and T.~Januszkiewicz.
\newblock Convex polytopes, {C}oxeter orbifolds and torus actions.
\newblock {\em Duke Math. J.}, 62(2):417--451, 1991.

\bibitem{dranishnikov2009lusternik}
A.~Dranishnikov.
\newblock On the lusternik-schnirelmann category of spaces with 2-dimensional
  fundamental group.
\newblock {\em Proceedings of the American Mathematical Society},
  137(4):1489--1497, 2009.

\bibitem{strongeqtc}
A.~Dranishnikov.
\newblock On topological complexity of twisted products.
\newblock {\em Topology Appl.}, 179:74--80, 2015.

\bibitem{Far}
M.~Farber.
\newblock Topological complexity of motion planning.
\newblock {\em Discrete Comput. Geom.}, 29(2):211--221, 2003.

\bibitem{F-G-L-O}
M.~Farber, M.~Grant, G.~Lupton, and J.~Oprea.
\newblock An upper bound for topological complexity.
\newblock {\em Topology Appl.}, 255:109--125, 2019.

\bibitem{Fin}
R.~Fintushel.
\newblock Classification of circle actions on {$4$}-manifolds.
\newblock {\em Trans. Amer. Math. Soc.}, 242:377--390, 1978.

\bibitem{Fra}
M.~Franz.
\newblock The cohomology rings of smooth toric varieties and quotients of
  moment-angle complexes.
\newblock {\em Geom. Topol.}, 25(4):2109--2144, 2021.

\bibitem{Gonzalez-Velasco}
J.~Gonz\'{a}lez and M.~Velasco.
\newblock Complex-projective and lens product spaces.
\newblock {\em Bol. Soc. Mat. Mex. (3)}, 20(2):319--333, 2014.

\bibitem{Kol}
J.~Koll\'{a}r.
\newblock Circle actions on simply connected 5-manifolds.
\newblock {\em Topology}, 45(3):643--671, 2006.

\bibitem{eqlscategory}
W.~Marzantowicz.
\newblock A {$G$}-{L}usternik-{S}chnirelman category of space with an action of
  a compact {L}ie group.
\newblock {\em Topology}, 28(4):403--412, 1989.

\bibitem{Mcc}
J.~McCleary.
\newblock {\em A user's guide to spectral sequences}, volume~58 of {\em
  Cambridge Studies in Advanced Mathematics}.
\newblock Cambridge University Press, Cambridge, second edition, 2001.

\bibitem{MiSt}
J.~W. Milnor and J.~D. Stasheff.
\newblock {\em Characteristic classes}.
\newblock Annals of Mathematics Studies, No. 76. Princeton University Press,
  Princeton, NJ; University of Tokyo Press, Tokyo, 1974.

\bibitem{Naskar}
B.~Naskar and S.~Sarkar.
\newblock On {LS}-category and topological complexity of some fiber bundles and
  {D}old manifolds.
\newblock {\em Topology Appl.}, 284:107367, 14, 2020.

\bibitem{Plo}
S.~Plotnick.
\newblock Circle actions and fundamental groups for homology {$4$}-spheres.
\newblock {\em Trans. Amer. Math. Soc.}, 273(1):393--404, 1982.

\bibitem{Ray}
F.~Raymond.
\newblock Classification of the actions of the circle on {$3$}-manifolds.
\newblock {\em Trans. Amer. Math. Soc.}, 131:51--78, 1968.
\end{thebibliography}
\end{document}